\documentclass[12pt]{amsart}

\usepackage{amsmath}
\usepackage{amsfonts}

\usepackage{graphicx}
\usepackage{amsthm}
\usepackage{amstext}
\usepackage{amssymb}
\usepackage{amsrefs}

\newcommand{\bm}[1]{\mbox{\boldmath $#1$}}

\def\nablah{\nabla^{\Sigma}}
\def\F{\mathfrak F}

\newcounter{mnotecount}

\newcommand{\mnotex}[1]
{\protect{\stepcounter{mnotecount}}$^{\mbox{\footnotesize $\bullet$\themnotecount}}$ 
\marginpar{
\raggedright\tiny\em
$\!\!\!\!\!\!\,\bullet $\themnotecount: #1} }

\numberwithin{equation}{section}
\numberwithin{figure}{section}

\newtheorem{theorem}{Theorem}[section]
\newtheorem{lemma}[theorem]{Lemma}
\theoremstyle{definition}
\newtheorem{definition}[theorem]{Definition}
\newtheorem{example}[theorem]{Example}

\theoremstyle{remark}
\newtheorem{remark}[theorem]{Remark}
\numberwithin{equation}{section}
\theoremstyle{plain}

\newtheorem{corollary}[theorem]{Corollary}

\newtheorem{proposition}[theorem]{Proposition}

\renewcommand{\div}{\text{div}}
\newcommand{\divm}{\text{\emph{div}}}

\DeclareMathOperator{\Ric}{Ric}
\DeclareMathOperator{\tr}{tr}

\newcommand{\II}{\operatorname{{I\hspace{-0.2mm}I}}}

\begin{document}

\begin{abstract}
In this sequel paper we give a shorter, second proof of the monotonicity of the Hawking mass for time flat surfaces under spacelike uniformly area expanding flows in spacetimes that satisfy the dominant energy condition.  We also include a third proof which builds on a known formula and describe a class of sufficient conditions of divergence type
for the monotonicity of the Hawking mass. 
These flows of surfaces may have connections to the problem in general relativity of bounding the total mass of a spacetime from below by the quasi-local mass of spacelike 2-surfaces in the spacetime. 

\end{abstract}

\title[Time flat surfaces and  the monotonicity of the Hawking mass II]{Time flat surfaces and  the monotonicity of the spacetime Hawking mass II}
\author{Hubert L. Bray}
\address{Dept. of Mathematics,
Duke University,
Durham, NC 27708}
\email{bray@math.duke.edu}

\author{Jeffrey L. Jauregui}
\address{Dept. of Mathematics,
Union College,
Schenectady, NY 12308}
\email{jaureguj@union.edu}

\author{Marc Mars}
\address{Instituto de F\'{\i}sica Fundamental y Matematicas, Universidad de Salamanca, Plaza de la Merced s/n, 37008 Salamanca, Spain}
\email{marc@usal.es}

\date{\today}
\maketitle

\section{Introduction}

A spacetime 
is defined to be a four-dimensional smooth 
manifold equipped with a metric $\langle \cdot, \cdot \rangle$ of Lorentzian signature $(-,+,+,+)$. We assume that the spacetime
is time oriented, i.e.
admits a nowhere zero timelike vector field, defined to be 
future-pointing.

\begin{definition}
An {\it admissible} surface $\Sigma$
is a smooth, closed,  spacelike
surface embedded in a spacetime such that the mean curvature vector
$\vec{H}$ is everywhere spacelike.
\end{definition}

Our conventions for the second fundamental form $\vec \II$ and mean curvature
$\vec{H}$ of $\Sigma$ are 
\begin{equation*}
   \vec\II(W,X) = \mbox{nor}_\Sigma \left( \nabla_W X \right),
 \quad \quad \vec{H} = \tr_{\Sigma} \vec{\II},
\end{equation*}
for vectors $W$ and $X$ tangent to $\Sigma$, where $\nabla$ is the Levi-Civita connection of the spacetime.

\begin{definition}
A normal vector field $\vec{\omega}$ along an admissible surface $\Sigma$ is 
\emph{inward}-spacelike (resp. -achronal)   if $\vec{\omega}$ is everywhere
spacelike (resp. non-timelike)
and $\langle \vec{\omega}, \vec{H} \rangle  >0$. It is called 
\emph{outward}  if 
$-\vec{\omega}$  is inward.
\end{definition}

An admissible surface $\Sigma$ has trivial normal bundle $T^\perp\Sigma$.
Define $\vec{H}^{\perp}$
along $\Sigma$ to be the unique past-directed normal vector orthogonal
to $\vec{H}$ with $\langle  \vec{H}^{\perp}, \vec{H}^{\perp} \rangle =
- \langle \vec{H},\vec{H} \rangle $ and the orientation on the normal bundle so that
$\{ \vec H, \vec H^{\perp} \}$ is positively oriented. We denote by
$\bm{\eta}^{\perp}$ the corresponding volume form on $T^\perp \Sigma$.

\begin{definition}
\label{perp}
Given a positively oriented orthonormal basis $\{ \vec e_r, \vec e_t \}$
with $\vec e_r$ spacelike and $\vec e_t$ timelike and any normal vector
$\vec{v} = a \vec{e}_{r} + b \vec{e}_{t}$, define
$\vec{v}^\perp = b \vec{e}_{r} + a \vec{e}_{t}$.
\end{definition}
Note $\vec{H}^{\perp}$ as defined before is consistent
with Definition \ref{perp}. An equivalent definition 
of $\vec{v}^{\perp}$ can be given in terms of the
Hodge dual with respect to the volume form. More specifically, for any
normal vector $\vec{v}$, $\vec{v}^{\perp}$ is the unique normal vector
satisfying
\begin{eqnarray*}
\bm{\eta}^{\perp} (\vec{w},\vec{v}^{\perp}) = \langle \vec{w},\vec{v} \rangle, 
\quad \quad \forall \; \vec{w} \, \,\, \mbox{normal to } \Sigma.
\end{eqnarray*}
This expression shows in particular
that the definition of $\vec v^{\perp}$ is independent
of the choice of orthonormal basis and that $\vec v^{\perp}$
is the Lorentzian version of a $90$ degree rotation.  Also note that
$\vec{v}$ and $\vec{v}^\perp$ are orthogonal and have the same length,
but if one is spacelike, then the other one is timelike. Note also that 
for outward-spacelike (-achronal) $\vec{v}$, $\vec{v}^{\perp}$ is future-timelike
(-causal).

\begin{definition}
A {\it regular} family of surfaces is defined to be a smooth family of 
admissible surfaces.
\end{definition}

\begin{definition}
A {\it uniformly area expanding} family of surfaces is defined to be a regular family of surfaces  such that the rate of change of the area density of each surface is itself, 
when the flow velocity is taken to be orthogonal to each surface.
\end{definition}

\begin{definition}
The {\it inverse mean curvature vector} of an admissible surface is defined to be $\vec{I} = \frac{-\vec{H}}{\langle \vec{H}, \vec{H} \rangle}$.  
\end{definition}
Note $\vec I$ is outward-spacelike.
\begin{lemma}
For any uniformly area expanding family of surfaces, the orthogonal flow velocity (i.e., the projection of the flow velocity onto the normal bundle) may be expressed as
\begin{equation*}
   \vec{\xi} = \vec{I} + \beta \vec{I}^\perp
\end{equation*}
for some smooth function $\beta$.
\end{lemma}
\begin{proof}
For a regular family of surfaces, $\vec{I}$ is outward-spacelike and $\vec{I}^\perp$ is future-timelike so that together they span the $(1+1)$-dimensional normal bundle.  Hence, the orthogonal flow velocity may be expressed as $\vec{\xi} = \alpha \vec{I} + \beta \vec{I}^\perp$. 
By the first variation formula for area, $\dot{dA} = \langle -\vec{H}, \vec{\xi} \rangle dA$, where $dA$ is the area density, 
so $\alpha = 1$.  Note that the flow is spacelike if and only if $|\beta| < 1$.
\end{proof}

\begin{definition} \label{def:Hawking}
The Hawking mass \cite{hawking} of a smooth spacelike surface $\Sigma$ in a spacetime is defined to be
\begin{equation*}
\label{eqn_hawking}
m_H(\Sigma) = \sqrt{\frac{|\Sigma|}{16\pi}} \left(1 - \frac{1}{16\pi} \int_{\Sigma} \langle \vec H, \vec H \rangle dA\right)
\end{equation*}
where $\vec H$ is the mean curvature vector, $dA$ is the area density,
and $|\Sigma|$ is the area of $\Sigma$. 
\end{definition}

The Hawking mass of a surface defines a notion of how much mass is inside the surface and is an example of a quasi-local mass functional.  For example, the Hawking mass of a spherically symmetric sphere in the Schwarzschild spacetime of mass $m$ is precisely $m$, which is one of the main motivations for this definition.  

In other examples, however, the Hawking mass may greatly overestimate or underestimate any reasonable notion of how much mass is inside a surface.  For example, consider a round sphere in the $t=0$ slice of the Minkowski spacetime (which represents vacuum).  This sphere has zero Hawking mass, which is reasonable.  The Gauss--Bonnet theorem may be used to prove that any closed, connected  surface in the $t=0$ hypersurface has nonpositive Hawking mass, and negative Hawking mass unless the surface is a round sphere.  This beautiful fact implies that the total mass of the $t=0$ hypersurface of the Minkowski spacetime, which is zero, is bounded below by the Hawking mass of any connected closed surface in this hypersurface.  

Even more remarkably, Huisken and Ilmanen proved that the total mass of a hypersurface with zero second fundamental form in a spacetime with nonnegative energy density is bounded below by the Hawking mass of any connected surface which bounds a finite region in the hypersurface and is not enclosed by a surface of equal or less area \cite{HI}.  This highly nontrivial result relies on a monotonicity formula for the Hawking mass \cite{geroch, jang_wald} under inverse mean curvature flow, as well as the corresponding existence and asymptotics results for the flow \cite{HI}.  

A natural question, then, is to wonder if a similar result could be true for the Hawking mass of surfaces in a spacetime that are not necessarily contained in a hypersurface with zero second fundamental form.  Returning to the Minkowski spacetime example, explicit calculation shows that adding ``squiggles'' to a round sphere in the $t=0$ slice in timelike directions can increase the Hawking mass.  In particular, the Hawking mass can be positive and hence too large to be a lower bound for the total mass of the spacetime, which is zero.  Thus, if we want the Hawking mass of a surface to be a lower bound for the total mass of a spacetime, we cannot allow surfaces with arbitrary timelike squiggles in them.  In the prequel, the \emph{time flat} condition was suggested as a possible means of ruling out such surfaces \cite{TF}.

\begin{definition}\label{def:alpha}
Given an admissible surface $\Sigma$ with mean curvature vector $\vec H$, let 
$\vec\nu_H = -\frac{\vec H}{|\vec H|}$  be the outward-spacelike unit vector parallel to 
$\vec H$.  The induced connection on $T^\perp \Sigma$ is characterized by the connection 1-form 
$\alpha_H(X) = \langle \nabla_X \vec\nu_H, \vec\nu_H^\perp \rangle$, 
where $X$ is any tangent vector $X$ to $\Sigma$.
\end{definition}

The Bartnik data \cite{bartnik} of a surface is equivalent to $(\Sigma, g|_\Sigma, H, \alpha_H)$, where $H = |\vec H|$, plus specifying the angle between the hypersurface and the mean curvature vector $\vec H$.  (For more discussion of Bartnik data, see \cite{DK}.)
Hence, this next definition is stated entirely in terms of the Bartnik data.  

\begin{definition}
A {\it time flat surface} $\Sigma$ is an admissible surface such that $\div_{\Sigma}(\alpha_H) = 0$.
\end{definition}

Note that $\div_\Sigma(\cdot)$ is the divergence on $\Sigma$.
The time flat condition is equivalent to the statement that the mean curvature vector 
$\vec H$, ignoring its length, points ``straight in,'' 
in a  reasonable sense described precisely in the prequel.  
Another important virtue of time flat surfaces is the following.

\begin{theorem}\label{thm:increasing} [Corollary 1.4 of \cite{TF}]
For a uniformly area expanding family of connected time flat surfaces $\Sigma(s)$ whose orthogonal flow velocity is achronal
(e.g., spacelike)
in a spacetime satisfying the dominant energy condition,
\begin{equation*}
   \frac{d }{ds} \left(m_H(\Sigma(s)) \right) \ge 0.
\end{equation*}
\end{theorem}

An obvious corollary is the following.

\begin{corollary}
Given a uniformly area expanding family of connected time flat surfaces $\Sigma(s)$, $s \ge 0$, beginning at $\Sigma_0 = \Sigma(0)$, whose orthogonal flow velocity is achronal (e.g., spacelike) 
 in a spacetime satisfying the dominant energy condition such that 
\begin{equation}\label{eqn:asymptotics}
\lim_{s \rightarrow \infty} m_H(\Sigma(s)) = m_{ADM},
\end{equation} 
the total ADM mass of the spacetime, then
\begin{equation}\label{eqn:lowerbound}
   m_{ADM} \ge m_H(\Sigma_0).
\end{equation}
\end{corollary}

A general existence theory for ``uniformly area expanding time flat flow'' starting from an initial time flat surface $\Sigma_0$ is an important open problem.  Combined with asymptotic results along the lines of equation (\ref{eqn:asymptotics}), a result like equation (\ref{eqn:lowerbound}) would be possible.  There are many important related questions to study here.

Theorem \ref{thm:increasing} is a corollary to the main theorem of the prequel to this paper, stated next.  Each of the five lines in Theorem \ref{thm:main} is nonnegative, proving Theorem \ref{thm:increasing}.    Since the Euler characteristic of a connected surface does not exceed $2$, the first line is nonnegative.  Since
$-\vec H$ and $\vec \xi$ are both outward-achronal, 
$-\vec H^\perp$ and $\vec \xi^\perp$ are both future-causal
so that $G(-\vec H^\perp, \vec \xi^\perp) \ge 0$ by the dominant energy condition.  
Since the flow velocity $\vec \xi$ is achronal, $|\beta| \leq 1$
so that the two middle terms in lines 3 and 4 are controlled, making these two lines nonnegative.  Finally, the fifth line is zero by the time flat assumption.

\newpage

\begin{theorem} [Theorem 1.1 of \cite{TF}]
\label{thm:main}
Given a uniformly area expanding family of surfaces $\Sigma(s)$, 
\begin{align*}
\frac{\frac{d }{ds} \left(m_H(\Sigma(s)) \right)}
{\sqrt{\frac{|\Sigma(s)|}{(16\pi)^3}} } \;\;=\;\;
& \; 4\pi\left(2-\chi(\Sigma(s))\right) \\
&+ \int_{\Sigma(s)} 2 G(-\vec H^\perp, \vec \xi^\perp) \\
&+\int_{\Sigma(s)}  \left[|\mathring \II_r |^2 + 2\beta \langle  \mathring \II_r, \mathring \II_t \rangle + |\mathring \II_t|^2 \right] \\
&+\int_{\Sigma(s)} 2 \left[ \left| \frac{\nabla^\Sigma H}{H}\right|^2 + 2\beta\alpha_H\left(\frac{\nabla^\Sigma H}{H}\right) +|\alpha_H|^2 \right] \\
&+\int_{\Sigma(s)} 2 \beta \cdot \divm_{\Sigma(s)}(\alpha_H)
 \end{align*}
where 
\begin{itemize}
\item the orthogonal flow velocity is $\vec{\xi} = \vec{I} + \beta \vec{I}^\perp$,
\item $\chi(\Sigma(s))$ is the Euler characteristic of $\Sigma(s)$, 
\item $G = \Ric - S \langle \;, \;\rangle$ is the Einstein curvature tensor of the spacetime, 
\item $H = |\vec H|$ is the length of the mean curvature vector $\vec H$ of $\Sigma(s)$,
\item $\vec\nu_H = -\frac{\vec H}{|\vec H|}$ defines the unit outward direction parallel to $\vec H$,
\item $\II_r = -\langle \vec \II, \vec\nu_H\rangle$ and 
$\II_t = -\langle \vec \II, \vec\nu_H^\perp\rangle$ are the components of the second fundamental form $\vec \II$ of $\Sigma(s)$ in the directions parallel and perpendicular to the mean curvature vector, and $\mathring \II_r$ and $\mathring \II_t$ are their traceless parts, respectively,
\item $\alpha_H(X) = \langle \nabla_X \vec\nu_H, \vec\nu_H^\perp \rangle$ for any tangent vector $X$ to $\Sigma(s)$,
\end{itemize}
and the area density $dA$ of $\Sigma(s)$ has been suppressed for convenience.
\end{theorem}

The  purpose of this sequel paper is twofold.  First, we give  two more proofs of the above theorem  that provide valuable additional insight (sections \ref{sec:plane_cylinder} and \ref{sec:bhms}).  One approach involves computing the variation of the Hawking mass separately in spacelike and timelike directions, while the other derives the equation from a known formula of Bray, Hayward, Mars, and Simon \cite{BHMS}.  Second, we investigate conditions on the family of surfaces that guarantee monotonicity of the Hawking mass (section \ref{sec:monotonicity}).

\vspace{2mm}
\paragraph{\emph{Acknowledgements:}} 
H.B. was supported in part by NSF grant \#DMS-1007063. M.M. acknowledges financial support under the projects  FIS2012-30926
(MICINN) and P09-FQM-4496 (Junta de Andaluc\'{\i}a and FEDER funds).

\section{The Plane / Cylinder Derivation}
\label{sec:plane_cylinder}

The problem of computing the rate of change of the Hawking mass 
\begin{equation*}
   \dot{m} = \frac{d}{ds}m_H(\Sigma(s))|_{s=0}
\end{equation*}
when flowing in the direction 
$\vec{\xi} = \vec{I} + \beta \vec{I}^\perp$ may be separated into two contributions, 
\begin{equation*}
   \vec{\xi}_r = \vec{I} \;\;\;\;\;\mbox{ and }\;\;\;\;\; \vec{\xi}_t = \beta \vec{I}^\perp
\end{equation*}
so that $\vec{\xi} = \vec{\xi}_r + \vec{\xi}_t$.  Decomposing into contributions parallel and perpendicular to the mean curvature vector (which we will sometimes call the radial and time directions, respectively) is  the main new idea that leads to the monotonicity formula in Theorem \ref{thm:main}.  The overall rate of change of the Hawking mass 
\begin{equation}
 \label{eqn_m_dot}
   \dot{m} = \dot{m}_r + \dot{m}_t
\end{equation}
will then be the sum of these two separate contributions.  

This idea is depicted in figure \ref{fig:cylinder} where for a short amount of flow time $s$ the flow in the radial direction locally sweeps out a spacelike 3-plane and the flow in the time direction locally sweeps out a (2+1)-dimensional cylinder, at least qualitatively.  The actual flows which sweep out the plane and the cylinder need only agree with $\vec{\xi}_r$ and $\vec{\xi}_t$ on the initial surface $\Sigma$ at $s=0$.  Theorem \ref{thm:main} then follows from the next two theorems and equation (\ref{eqn_m_dot}).

\begin{theorem}\label{thm:plane} (The plane theorem)
The initial rate of change of the Hawking mass of a regular family of surfaces beginning with $\Sigma$ and flowing with initial velocity $\vec{\xi}_r = \vec{I}$ is 
\begin{equation*}
\frac{\dot{m}_r}
{\sqrt{\frac{|\Sigma|}{(16\pi)^3}} } 
= 4\pi (2 - \chi(\Sigma)) 
+ \int_{\Sigma} 2 G(-\vec H^\perp, \vec I^\perp) 
+ \left[|\mathring \II_r |^2 
+ |\mathring \II_t|^2 \right] 
+ 2 \left[ \left| \frac{\nabla^\Sigma H}{H}\right|^2  +|\alpha_H|^2 \right] 
\end{equation*}
with everything defined as before in Theorem \ref{thm:main}.
\end{theorem}

While the above result follows from previous works such as \cite{frauendiener}, we include the proof below for clarity and completeness, as well as to establish notation.

\begin{theorem}\label{thm:cylinder} (The cylinder theorem)
The initial rate of change of the Hawking mass of a regular family of surfaces beginning with $\Sigma$ and flowing with initial velocity $\vec{\xi}_t = \beta \vec{I}^\perp$ is 
\begin{equation*}
\frac{\dot{m}_t}
{\sqrt{\frac{|\Sigma|}{(16\pi)^3}} } 
= \int_{\Sigma} 2 G(-\vec H^\perp, \beta \vec I) 
+ 2\beta \langle  \mathring \II_r, \mathring \II_t\rangle
+ 4\beta\alpha_H\left(\frac{\nabla^\Sigma H}{H}\right)  
+ 2 \beta \cdot \divm_{\Sigma}(\alpha_H)
\end{equation*}
with everything defined as before in Theorem \ref{thm:main}.
\end{theorem}

\begin{figure}
\begin{center}
\includegraphics[width=6.5in]{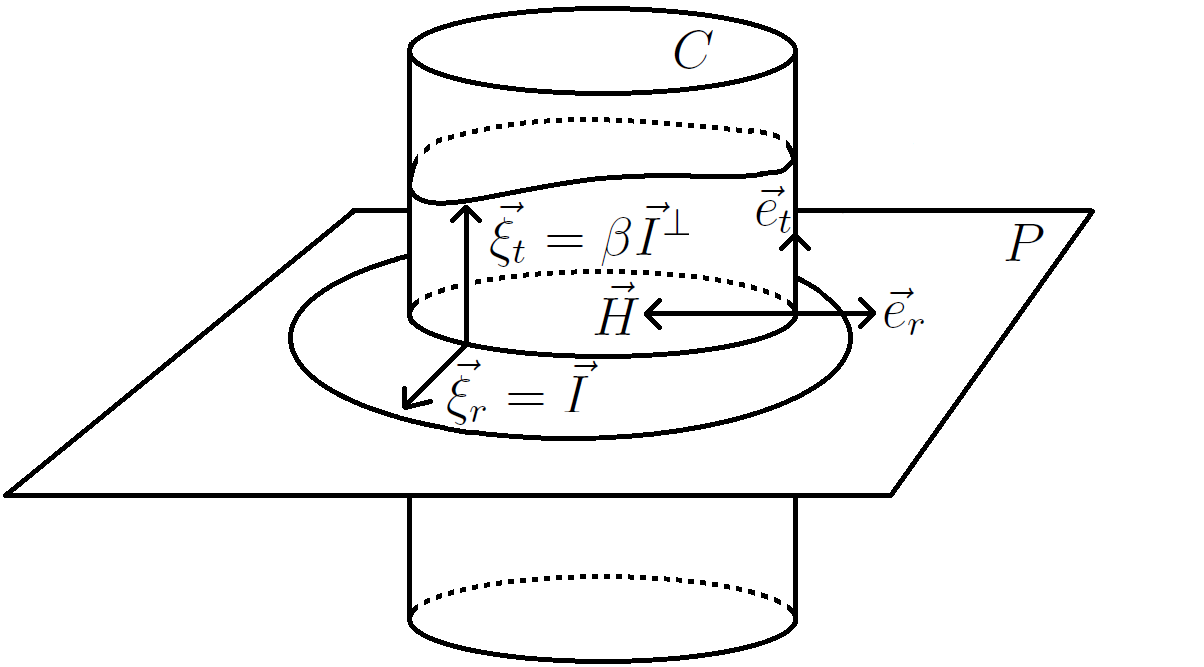}
\caption{The surface $\Sigma^2$ is at the intersection of the ``plane'' $P^3$ and the ``cylinder'' $C^{2,1}$ generated by flowing $\Sigma^2$ by $\vec{\xi}_r = \vec{I}$ and $\vec{\xi}_t = \beta \vec{I}^\perp$.
 \label{fig:cylinder}}
\end{center}
\end{figure}

\subsection{Proof of the plane theorem}

In this subsection we prove Theorem \ref{thm:plane}.  
Let $P$ be the spacelike hypersurface swept out by the $\Sigma(s)$ beginning at $\Sigma = \Sigma(0)$, and suppose the flow velocity at $s=0$ is $\vec{\xi}_r = \vec{I}$.  By the first variation formula for area, 
\begin{eqnarray}
\dot{dA} |_{s=0} = \langle -\vec{H}, \vec{\xi_r} \rangle dA |_{s=0} &=& dA \label{eqn:areaform} \\
\frac{d}{ds} |\Sigma(s)| |_{s=0} &=& |\Sigma|. \label{eqn:area}
\end{eqnarray}

Let $\vec e_t$ be the future unit normal to $P$ and $\vec e_r$ be the outward unit normal to $\Sigma(s)$ in $P$.  Note that at $s=0$, $\vec e_t = \vec\nu_H^\perp$ and $\vec e_r = \vec\nu_H$ on the initial surface $\Sigma$.
Next, let the second fundamental form of $P$ in the spacetime to be
\begin{equation*} 
   \vec\II_P(Y,Z) = \mbox{nor}_P \left( \nabla_Y Z \right) = k(Y, Z) \vec e_t
\end{equation*}
for vectors $Y$ and $Z$ tangent to $P$, where $k$ is scalar-valued.
Finally, define $H = - \langle \vec H, \vec e_r \rangle$ to be the scalar-valued mean curvature of $\Sigma(s)$ in $P$.  Since $\vec e_r = \vec\nu_H$ at $s=0$, $H = |\vec H|$ on the initial surface $\Sigma$.

Since $\vec e_t$ and $\vec e_r$ are an orthonormal basis for the normal bundle to $\Sigma(s)$, 
\begin{eqnarray*}
   \vec H = \tr_\Sigma (\vec\II) &=& 
   \langle \vec H, \vec e_r \rangle \vec e_r - \langle \tr_\Sigma (\vec\II), \vec e_t \rangle \vec e_t \\
   &=& -H \vec e_r - \langle \tr_\Sigma (\vec\II_P), \vec e_t \rangle \vec e_t \\
   &=& -H \vec e_r + \tr_\Sigma (k) \vec e_t,
\end{eqnarray*}
since $\vec \II_P - \vec \II$, restricted to
vectors tangent to $\Sigma$, 
is orthogonal to $\vec e_t$.  Here, $\tr_\Sigma(\cdot)$ is the trace with respect to $\Sigma$. Hence, 
\begin{equation*}
   \langle \vec H, \vec H \rangle = H^2 - (\tr_\Sigma (k))^2,
\end{equation*}
for all $s$, and when $s=0$,
\begin{equation}\label{eqn:zerotrace}
 \tr_\Sigma (k) = -\langle \vec H, \vec e_t \rangle = -\langle \vec H, \vec\nu_H^\perp \rangle  = 0
\end{equation}
on $\Sigma$.  Thus,
\begin{eqnarray}
   \left. \frac{d}{ds} \langle \vec H, \vec H \rangle \right|_{s=0} \nonumber
   &=& \left. 2H\left(\frac{d}{ds} H\right)   \right|_{s=0}
       - 2 \left. \tr_\Sigma (k) \left(\frac{d}{ds}\tr_\Sigma (k) \right)  \right|_{s=0} \\
   &=& \left. 2H\left(\frac{d}{ds} H\right) \right|_{s=0} .  \label{eqn:H2}
\end{eqnarray}

The initial speed of the flow is $|\vec\xi_r| = |\vec I| = \frac{1}{H}$, so
by the second variation formula 
\begin{equation}
   \left.\frac{d}{ds} H \right|_{s=0} = 
   -\Delta_\Sigma \left(\frac{1}{H}\right) 
   - |\!\II_r\!|^2 \left(\frac{1}{H}\right)
   - \Ric_P(\vec e_r, \vec e_r) \left(\frac{1}{H}\right), \label{eqn:SVF}
\end{equation}
where $\Ric_P$ is the Ricci curvature of $P$ and $\II_r = -\langle \vec \II, \vec\nu_H\rangle$ agrees with the scalar-valued second fundamental form of $\Sigma$ in $P$.  The Gauss equation traced twice over $\Sigma$ inside $P$ is
\begin{equation}
   2\Ric_P(\vec e_r, \vec e_r) = R - 2K + H^2 - |\!\II_r\!|^2, \label{eqn:G1}
\end{equation}
where $R$ is the scalar curvature of $P$ and $K$ is the Gauss curvature of $\Sigma$.
The Gauss equation traced twice over $P$ inside the spacetime is 
\begin{equation}
   R = 2 G(\vec e_t, \vec e_t) + |k|_P^2 - (\tr_P \!k)^2, \label{eqn:G2}
\end{equation}
where $G$ is the Einstein curvature tensor of the spacetime and $|\cdot|_P$ and $\tr_P(\cdot)$ are the tensor norm and trace with respect to $P$.

By choosing an orthonormal basis for $TP$ along $\Sigma$ that includes $\vec \nu_H$, we conclude that 
\begin{eqnarray*}
   |k|_P^2 &=& k(\vec \nu_H, \vec \nu_H)^2 + 2 |k(\vec \nu_H, \cdot|_\Sigma) |^2 
                      + |k|_\Sigma^2 \\
   (\tr_P k)^2 &=& (k(\vec \nu_H, \vec \nu_H) + \tr_\Sigma k)^2 \\
                   &=& k(\vec \nu_H, \vec \nu_H)^2
\end{eqnarray*}
by equation (\ref{eqn:zerotrace}), where `` $\cdot|_\Sigma$'' indicates restricting the domain to vectors tangent to $\Sigma$
and  $|\cdot|_\Sigma $ is the tensor norm computed using only directions tangent to $\Sigma$. 
Then: 
\begin{eqnarray}
   |k|_P^2 - (\tr_P k)^2 &=& 2 |k(\vec \nu_H, \cdot|_\Sigma) |^2 + |k|_\Sigma^2 \nonumber\\
                                 &=& 2 |\alpha_H|^2 + |\!\II_t\!|^2  \label{eqn:k2}
\end{eqnarray}
since $\alpha_H(X) = \langle \nabla_X \vec\nu_H, \vec\nu_H^\perp \rangle
                          = \langle \vec\II_P(X, \vec\nu_H), \vec e_t\rangle 
                          = - k(\vec \nu_H, X)$, and where
$\II_t = -\langle \vec \II, \vec\nu_H^\perp\rangle$.

We may now compute the initial rate of change of the Hawking mass in Definition \ref{def:Hawking} using equations (\ref{eqn:areaform})--(\ref{eqn:area}) and  (\ref{eqn:H2})--(\ref{eqn:k2}) to obtain
\begin{equation*}
\frac{\dot{m}_r}
{\sqrt{\frac{|\Sigma|}{(16\pi)^3}} } 
= 8\pi 
+ \int_{\Sigma} -2K 
+ 2 G(\vec e_t,\vec e_t) 
+ \left(\!|\II_r\!|^2 - \frac12 H^2\right)
+ |\!\II_t\!|^2
+ 2H \Delta_\Sigma\left(\frac{1}{H}\right)
+2|\alpha_H|^2.
\end{equation*}
The following observations complete the proof of the plane theorem.
\begin{itemize}
\item By the Gauss-Bonnet theorem, $\int_{\Sigma} K = 2\pi \chi(\Sigma)$.  
\item By tensorality, 
$G(\vec e_t, \vec e_t) = G(H\vec e_t, \frac{1}{H}\vec e_t) = G(-\vec H^\perp, \vec I^\perp)$. 
\item Since the trace of $\II_r$ is $H$, its traceless part $\mathring\II_r $ equals  $\II_r - \frac12 H g|_\Sigma$, where $g|_{\Sigma}$
is the restriction of the metric to $T\Sigma$.  Thus:
\begin{eqnarray*}
   |\mathring \II_r |^2 &=& \left\langle \II_r - \frac12 H g|_\Sigma, \II_r - \frac12 H  g|_\Sigma \right\rangle \\
   &=& |\!\II_r\! |^2 -2 \cdot \frac12 H^2 + \frac14 H^2 \cdot 2 \\
   &=& |\!\II_r\! |^2 - \frac12 H^2.
\end{eqnarray*}
\item By equation \ref{eqn:zerotrace}, $\II_t$ is traceless, so $|\mathring\II_t|^2 = |\!\II_t\!|^2$.
\item Finally, integrating by parts implies  
\begin{equation*}
\int_{\Sigma} H \Delta_\Sigma\left(\frac{1}{H}\right) =
\int_{\Sigma} - \left\langle \nabla^\Sigma H, \nabla^\Sigma \left(\frac{1}{H}\right) \right\rangle =
\int_{\Sigma} \left| \frac{\nabla^\Sigma H}{H}\right|^2.
\end{equation*}
\end{itemize}

\subsection{Proof of the cylinder theorem}

In this subsection we prove Theorem \ref{thm:cylinder}.
Let $C$ be the $(2+1)$-dimensional hypersurface swept out by the timelike flow $\Sigma(s)$ beginning at $\Sigma = \Sigma(0)$ whose flow velocity at $s=0$ is  $\vec\xi_t = \beta \vec I^\perp$.  

By the first variation formula for area, 
\begin{eqnarray}
\dot{dA} |_{s=0} = \langle -\vec{H}, \vec{\xi_t} \rangle dA |_{s=0} &=& 0 \label{eqn:areaform2} \\
\frac{d}{ds} |\Sigma(s)| |_{s=0} &=& 0, \label{eqn:area2}
\end{eqnarray}
which is why $C$ is depicted as a cylinder in figure \ref{fig:cylinder}.

Next we need to adapt our definitions from the previous subsection to be defined on $C$.  Quantities defined on both $C$ and $P$ need only agree where they intersect, namely on the surface $\Sigma$.  
  
Let $\vec e_r$ be the outward unit normal to $C$ and $\vec e_t$ be the future unit normal to $\Sigma(s)$ in $C$.  Note that at $s=0$, $\vec e_r = \vec\nu_H$ and $\vec e_t = \vec\nu_H^\perp$ on $\Sigma$, as before.
Define the second fundamental form of $C$ in the spacetime to be 
\begin{equation*} 
   \vec\II_C(Y,Z) = \mbox{nor}_C \left( \nabla_Y Z \right) = \tilde k(Y, Z) \vec e_r
\end{equation*}
for vectors $Y$ and $Z$ tangent to $C$, where $\tilde k$ is scalar-valued.
Let $p = (\tr_C \! \tilde k) g_C - \tilde k$ on $C$, where $\tr_C$ is the trace with respect to the induced metric $g_C$ on $C$ of signature $(-,+,+)$.

Note that on $C$,
\begin{eqnarray}
p(\vec e_t, \vec e_t) &=&  \left(\tr_{\Sigma(s)}(\tilde k) -\tilde k(\vec e_t, \vec e_t)\right) (-1) - \tilde k(\vec e_t, \vec e_t) \\ &=& -\tr_{\Sigma(s)}(\tilde k), \nonumber 
\end{eqnarray}
for all $s$, and at $s=0$:
\begin{eqnarray}
p(\vec e_t, X) &=& -\tilde k(\vec e_t, X) = -\langle \vec\II_C(\vec e_t, X), \vec e_r \rangle
                   = -\langle \nabla_X \vec e_t, \vec e_r  \rangle 
                   = \langle \nabla_X \vec e_r, \vec e_t  \rangle \label{eqn:palpha} \\
&=& \alpha_H(X) \nonumber
\end{eqnarray}
where $X$ is any tangent vector of $\Sigma$.  Recall from Definition \ref{def:alpha} that $\alpha_H$ is the connection 1-form with respect to $\vec\nu_H$ for the normal bundle of $\Sigma(s)$.

\begin{definition}
Define $H = - \langle \vec H, \vec e_r \rangle$ as before and define $H_C = - \langle \vec H, \vec e_t \rangle$ to be the scalar-valued mean curvature of $\Sigma(s)$ in $C$.  
Note that $H = |\vec H|$ and $H_C = 0$ on $\Sigma$.
\end{definition}

Since $\vec e_t$ and $\vec e_r$ are an orthonormal basis for the normal bundle to $\Sigma(s)$, 
\begin{eqnarray*}
   \vec H = \tr_{\Sigma(s)} (\vec\II) &=& 
   \langle \tr_{\Sigma(s)} (\vec\II), \vec e_r \rangle \vec e_r - \langle \vec H , \vec e_t \rangle \vec e_t \\
   &=& \langle \tr_{\Sigma(s)} (\vec\II_C), \vec e_r \rangle \vec e_r + H_C \vec e_t\\
   &=& \tr_{\Sigma(s)}(\tilde k) \vec e_r + H_C \vec e_t \\
   &=& -p(\vec e_t, \vec e_t) \vec e_r + H_C \vec e_t,
\end{eqnarray*}
since $\vec \II_C - \vec \II$, restricted to
vectors tangent to $\Sigma(s)$,
 is orthogonal to $\vec e_r$. Hence, for all $s$,
\begin{equation}\label{eqn:zerotrace2}
  p(\vec e_t, \vec e_t) = -\langle \vec H, \vec e_r \rangle = H 
\end{equation}
and 
\begin{equation*}
   \langle \vec H, \vec H \rangle = p(\vec e_t, \vec e_t)^2 - H_C^2.
\end{equation*}
Thus,
\begin{eqnarray}
   \left.\frac{d}{ds} \langle \vec H, \vec H \rangle \right|_{s=0} \nonumber
   &=& 2p(\vec e_t, \vec e_t)\left.\left(\frac{d}{ds} p(\vec e_t, \vec e_t)\right) \right|_{s=0}
       - \left.2 H_C \left(\frac{d}{ds} H_C \right)  \right|_{s=0} \\
   &=& \left.2H\left(\frac{d}{ds} p(\vec e_t, \vec e_t)\right) \right|_{s=0} .  \label{eqn:H22}
\end{eqnarray}
The initial speed of the flow is 
\begin{equation}\label{eqn:f}
f = \frac{\beta}{H}.
\end{equation}
Hence, the flow at $s=0$ is $\vec\xi_t = \beta \vec I^\perp = f\vec e_t$, so
\begin{eqnarray} 
\left.\frac{d}{ds} p(\vec e_t, \vec e_t) \right|_{s=0} 
&=& \left.f (\nabla^C_{\vec e_t} p)(\vec e_t, \vec e_t)\right|_{s=0} + \left.2p(\vec e_t, \nabla^C_{f\vec e_t} (\vec e_t))\right|_{s=0} \nonumber \\ \nonumber
&=& \left.f (\nabla^C_{\vec e_t} p)(\vec e_t, \vec e_t)\right|_{s=0} + \left.2p(\vec e_t, \nabla^\Sigma f)\right|_{s=0} \\
&=& \left.f (\nabla^C_{\vec e_t} p)(\vec e_t, \vec e_t)\right|_{s=0} + 2\alpha_H(\nabla^\Sigma f) \label{eqn:pdot}
\end{eqnarray}
by equation (\ref{eqn:palpha}), where $\nabla^C$ is the Levi-Civita connection on $C$. We justify next the 
expression $\nabla^C_{f \vec e_t} (\vec e_t) = \nabla^{\Sigma} f$
used in the second line. On open sets where $f=0$
this is obvious. At points where $f \neq 0$, 
let $s$ be the function on $C \setminus \{ f=0 \}$
taking the constant value $s$ on $\Sigma(s)$.
Since  $s$ is the flow parameter, we have $\vec\xi_t(s) = 1$ 
and hence $\nabla^C s = - f^{-1} \vec{e}_t$  because $\nabla^C s$ is parallel to $\vec{e}_t$ and
$1 = \vec{\xi}_t (s) = \langle \vec{\xi}_t, \nabla^C s \rangle = f \langle \vec{e}_t, \nabla^C s \rangle $.
Thus,  $|\nabla^C s|^2 = - f^{-2}$ and 
\begin{align*}
\nabla^C_{\vec e_t} \vec e_t & =
\nabla_{f \nabla^C s} (f \nabla^C s) = 
f (\nabla_{\nabla^C s} f) \nabla^C s + f^2 \nabla^C_{\nabla^C s} (\nabla^C s)\\ 
&= 
f (\nabla^C_{\nabla^C s} f) \nabla^C s + \frac{1}{2} f^2 \nabla^C ( |\nabla^C s|^2 ) \\
& =
f (\nabla^C_{\nabla^C s} f) \nabla^C s + f^{-1}  \nabla^C f.
\end{align*}
Since $\nabla^C_{\vec e_t} \vec e_t$ is tangent to $\Sigma(s)$, the claim follows.

The purpose of the next calculations is to find another expression for $f (\nabla^C_{\vec e_t} p)(\vec e_t, \vec e_t)$ on $\Sigma$. If we let $\{\vec e_1, \vec e_2\}$ be a local orthonormal frame for $T\Sigma$, 
\begin{align*}
\left(\div_C  (p)(\vec e_t) + (\nabla^C_{\vec e_t} p)(\vec e_t, \vec e_t)\right) - \div_{\Sigma}(\alpha_H)
&= \sum_{i=1,2} (\nabla^C_{\vec e_i} p)(\vec e_i, \vec e_t) - \sum_{i=1,2} ( \nabla^\Sigma_{\vec e_i} \alpha_H)(\vec e_i) \\
&= \sum_{i=1,2}\vec e_i(p(\vec e_i, \vec e_t)) -p(\nabla^C_{\vec e_i} \vec e_i, \vec e_t) - p(\vec e_i, \nabla^C_{\vec e_i} \vec e_t)\\
&\qquad - \sum_{i=1,2} \vec e_i (\alpha_H(\vec e_i)) - \alpha_H(\nabla^\Sigma_{\vec e_i} \vec e_i) \\
&= \sum_{i=1,2} -p(\nabla^C_{\vec e_i} \vec e_t, \vec e_i),
\end{align*}
by (\ref{eqn:palpha}) and since
$$\sum_{i=1,2} \nabla^C_{\vec e_i} \vec e_i - \nabla^\Sigma_{\vec e_i} \vec e_i = H_C \vec e_t = 0$$
 on $\Sigma$.
Substituting
\begin{equation*}
   \nabla^C_{\vec e_i} \vec e_t = \sum_{j=1,2} \langle \nabla^C_{\vec e_i} \vec e_t, \vec e_j \rangle \vec e_j
   = -\sum_{j=1,2} \langle \vec e_t, \nabla^C_{\vec e_i} \vec e_j \rangle \vec e_j
   = -\sum_{j=1,2} \langle \vec e_t, \vec\II(\vec e_i, \vec e_j) \rangle \vec e_j
   =  \sum_{j=1,2} \II_t(\vec e_i, \vec e_j) \vec e_j
\end{equation*}
we get 
\begin{equation*}
   \sum_{i=1,2} p(\nabla^C_{\vec e_i} \vec e_t, \vec e_i) = 
   \sum_{i,j=1,2} \II_t(\vec e_i, \vec e_j) p( \vec e_j, \vec e_i) = 
   \langle \II_t, p \rangle_\Sigma
\end{equation*}
Since 
$\tr_\Sigma( \II_t) = - \langle \tr_\Sigma(\vec \II), \vec\nu_H^\perp \rangle 
                        = -\langle \vec H, \vec\nu_H^\perp \rangle = 0$
on $\Sigma$, $\II_t$ is traceless and 
\begin{equation*}
\langle \II_t, p\rangle_\Sigma = \langle \mathring \II_t, p\rangle_\Sigma
= - \langle \mathring \II_t, \tilde k\rangle_\Sigma
=   \langle \mathring \II_t, \II_r\rangle
=   \langle  \mathring \II_t, \mathring \II_r\rangle,
\end{equation*}
since on $T\Sigma$, $\tilde k = \langle \vec \II_C, \vec e_r\rangle = \langle \vec \II, \vec e_r\rangle = -\II_r.$
Putting the previous four lines together, we have that 
\begin{equation}\label{eqn:divs} 
   (\nabla_{\vec e_t} p)(\vec e_t, \vec e_t) = 
   \div_{\Sigma}(\alpha_H) - \div_C  (p)(\vec e_t)
   - \langle \mathring \II_t, \mathring \II_r\rangle.
\end{equation}
Finally, the Codazzi equation traced over $C$ implies that 
\begin{equation}\label{eqn:Codazzi}
   \div_C  (p)(\vec e_t) = \Ric(\vec e_t, \vec e_r) = G(\vec e_r, \vec e_t),
\end{equation}
where $\Ric$ is the Ricci curvature and $G$ is the Einstein curvature of the spacetime.
We may now compute the initial rate of change of the Hawking mass in Definition \ref{def:Hawking} using equations (\ref{eqn:areaform2})--(\ref{eqn:area2}), (\ref{eqn:H22})--(\ref{eqn:Codazzi}) to get
\begin{eqnarray*}
\frac{\dot{m}_t}
{\sqrt{\frac{|\Sigma|}{(16\pi)^3}} } 
&=& -\int_{\Sigma} 2H\left\{2\alpha_H\left(\nabla^\Sigma \frac{\beta}{H}\right)  
+ \frac{\beta}{H} \left[ \div_{\Sigma}(\alpha_H) 
- \langle \mathring \II_t, \mathring\II_r\rangle 
- G(\vec e_t, \vec e_r) \right] \right\} \\
&=& \int_{\Sigma} -4\alpha_H(\nabla^\Sigma \beta) + 4\beta \alpha_H\left(\frac{\nabla^\Sigma H}{H}\right) 
       - 2\beta \left[ \div_{\Sigma}(\alpha_H) - \langle \mathring \II_t, \mathring \II_r\rangle 
- G(\vec e_t, \vec e_r) \right]  \\
&=& \int_{\Sigma} 2 \beta \cdot \div_{\Sigma}(\alpha_H)
+ 4\beta\alpha_H\left(\frac{\nabla^\Sigma H}{H}\right)  
+ 2\beta \langle  \mathring \II_t, \mathring \II_r\rangle
+ 2 G(-\vec H^\perp, \beta \vec I),
\end{eqnarray*}
having integrated by parts at the last step and used $G(\vec e_t, \vec e_r) = G(H \vec e_t, \frac1H\vec e_r) = G(-\vec H^\perp, \vec I)$.  This proves the cylinder theorem.
The main result, Theorem \ref{thm:main}, then follows immediately from the plane theorem (Theorem \ref{thm:plane}) and the cylinder theorem (Theorem \ref{thm:cylinder}).

\section{Alternate proof of the variation of the Hawking mass formula}
\label{sec:bhms}
In this section we provide an alternate proof of the main formula, Theorem \ref{thm:main}.  The key is a prior result of  Bray, Hayward, Mars, and Simon 
on the variation of the Hawking mass in an arbitrary flow direction $\vec \xi$ \cite{BHMS}.  
Without loss of generality, we may compute the derivative of the Hawking mass at flow time $s=0$.

\begin{theorem}[cf. Lemma 4 of \cite{BHMS}] 
\label{thm:BHMS}
Let $\Sigma(s)$ be a uniformly area expanding family of surfaces with velocity $\vec \xi$ at $s=0$.   Let $\Sigma_0=\Sigma(0)$, and assume the $\Sigma(s)$ are topologically spherical.  Then:
\begin{align*}
\frac{\frac{d }{ds} \left(m_H(\Sigma(s)) \right)}
{\sqrt{\frac{|\Sigma(s)|}{(16\pi)^3}} }\Bigg|_{s=0} &= \int_{\Sigma_0} \Big[2G(-\vec H^\perp, \vec \xi^\perp) + 16\pi\Theta^T + 16\pi \Theta^L - 2\divm_\Sigma(U)\langle \vec \xi, -\vec H^\perp\rangle\Big].
\end{align*}
\end{theorem}
The quantities $\Theta^T, \Theta^L,$ and $U$ are explained below.  We restrict to the case of topological spheres for simplicity and to be consistent with \cite{BHMS}.

\begin{remark}
The above formula is presented in a slightly different form than in \cite{BHMS}.  First, our definition of Hawking mass 
does not include a cosmological constant term
(cf. formula (4) in \cite{BHMS}).
Second, the sign convention for the second fundamental form and mean curvature vector in \cite{BHMS} are opposite those of the present paper and have been modified accordingly.
Third, the version of the formula presented above is simpler than that in \cite{BHMS} because we restrict to uniformly area expanding flows.
\end{remark}

We now find expressions for $U$, $\Theta^T$, and $\Theta^L$ separately in the cases $\vec \xi = \vec \xi_r = \vec I$ and $\vec \xi = \vec \xi_t = \beta \vec I^\perp.$

\vspace{2mm}
\paragraph{\emph{Computing $U$:}}
Let $\{ \vec l, \vec k \}$ be a positively oriented null basis of $T^\perp \Sigma_0$, so that
$$\vec \xi = A \vec l + B \vec k$$
for some functions $A,B$ on $\Sigma_0$.  When $\vec \xi$ is non-null,
$A$ and $B$ never vanish.
Define $\phi = -\langle \vec l, \vec k\rangle$, and
define $U$ to be the following 1-form on $\Sigma_0$:
$$U(X) = \frac{1}{2\phi}\left(\frac{\langle \vec l, \nabla_X(B\vec k) \rangle}{B} - \frac{\langle \vec k, \nabla_X(A\vec l)}{A}\right).$$
It is immediate to check that $U$ is independent of the choice of 
$\{\vec{l},\vec{k} \}$. 
Elementary calculations show
$$U(X) = \frac{1}{2} \left(\frac{D_X(A)}{A} - \frac{D_X(B)}{B}\right) + \frac{1}{2\phi}\left(\langle \nabla_X\vec k,\vec l\rangle - \langle\nabla_X \vec l, \vec k \rangle\right).$$
Let us fix the null frame
$$\vec l = \vec \nu_H + \vec \nu_H^\perp, \qquad\qquad \vec k = -\vec \nu_H+ \vec \nu_H^\perp.$$
Then $\phi=2$ and
$$\langle \nabla_X\vec k,\vec l\rangle = -\langle \nabla_X\vec l,\vec k\rangle =-2\alpha_H(X),$$
so that
$$U(X) = \frac{1}{2} \left(\frac{D_X(A)}{A} - \frac{D_X(B)}{B}\right) - \alpha_H(X).$$
First, suppose $\vec \xi = \vec \xi_r = \vec I$. Then 
$$\vec \xi = \frac{1}{H}\vec\nu_H = \frac{1}{2H} (\vec l - \vec k),$$ and
we have $A=-B$, which implies:
\begin{equation}
\label{eqn_U1}
U=-\alpha_H.
\end{equation}
Next, suppose $\vec \xi = \vec \xi_t = \beta \vec I^\perp$.
In the formula in Theorem \ref{thm:BHMS}, the term involving $U$ is multiplied by $\langle \vec \xi, -\vec H^\perp\rangle=-\beta$, so it suffices to define $U$ only on the set $\{ \beta \neq 0 \}$, where $\vec \xi$ is not null.
Now  $$\vec \xi = \beta \vec I^\perp = \frac{\beta}{H}\vec\nu_H^\perp = \frac{\beta}{2H} (\vec l + \vec k),$$
so that $A=B$ and, as before,
\begin{equation}
\label{eqn_U2}
U=-\alpha_H.
\end{equation}

\vspace{2mm}
\paragraph{\emph{Computing $\Theta^T$}}

Let $\vec \II^\circ$ be the trace-free part of the second fundamental form of $\Sigma_0$ in the spacetime.
Then $\Theta^T$ is defined by equations (14) of \cite{BHMS}:

\begin{align}
8\pi \Theta^T &=\langle -\vec \II^\circ_{ab}, -\vec H\rangle \langle \vec \xi, - (\vec \II^\circ)^{ab}\rangle - \frac{1}{2} \langle -\vec \II^\circ_{ab}, -(\vec \II^\circ)^{ab}\rangle \langle \vec \xi, -\vec H\rangle \label{eqn_14a}\\
&=\langle -\vec \II^\circ_{ab}, -\vec H^\perp\rangle \langle \vec \xi^\perp, - (\vec \II^\circ)^{ab}\rangle - \frac{1}{2} \langle -\vec \II^\circ_{ab}, -(\vec \II^\circ)^{ab}\rangle \langle \vec \xi^\perp, -\vec H^\perp\rangle, \label{eqn_14b}
\end{align}
where $a$ and $b$ are indices corresponding to a local orthonormal frame on $\Sigma$.

First, consider the case in which $\vec \xi = \vec I$.  We add (\ref{eqn_14a}) and (\ref{eqn_14b}), noting the two terms on the right cancel by the definition of $\;^\perp$:
\begin{align}
16\pi \Theta^T &= \langle \vec \II^\circ_{ab}, -\vec H\rangle \langle \vec \xi,  (\vec \II^\circ)^{ab}\rangle + \langle\vec\II^\circ_{ab}, -\vec H^\perp\rangle\langle\vec \xi^\perp, (\vec\II^\circ)^{ab}\rangle\nonumber\\
&= \langle \vec \II^\circ_{ab}, \vec \nu_H\rangle\langle (\vec \II^\circ)^{ab},\vec \nu_H\rangle + \langle \vec \II^\circ_{ab}, \vec \nu_H^\perp\rangle\langle (\vec \II^\circ)^{ab},\vec \nu_H^\perp\rangle \nonumber\\
&= |\mathring\II_{r}|^2+|\mathring\II_{t}|^2, \label{eqn_Theta_T1}
\end{align}
by definition of $\mathring\II_r$ and $\mathring\II_t$.

Second, suppose $\vec \xi = \beta \vec I^\perp$.  Beginning with (\ref{eqn_14a}), we have
\begin{align}
16\pi \Theta^T &=2\langle \vec \II^\circ_{ab}, -\vec H\rangle \langle \vec \xi,  (\vec \II^\circ)^{ab}\rangle\nonumber\\
&=2\beta \langle\vec \II^\circ_{ab}, \vec \nu_H\rangle \langle   (\vec \II^\circ)^{ab},\vec \nu_H^\perp\rangle\nonumber\\
&=2\beta \langle \mathring\II_{r}, \mathring\II_{t}\rangle.\label{eqn_Theta_T2}
\end{align}

\vspace{2mm}
\paragraph{\emph{Computing $\Theta^L$}} 
For non-null  $\vec \xi$, the expression $\Theta^L$ is defined by equation (23) of \cite{BHMS}:
\begin{equation}
8 \pi \Theta^L = (|U|^2 + | d \psi |^2  ) 
\langle \vec \xi, -\vec H \rangle
 -  2 \langle U, d \psi \rangle \langle \vec \xi, - \vec{H}^{\perp} \rangle 
\label{defThetaL}
\end{equation}
where
$$ e^{2 \psi } = | \langle \vec \xi, \vec \xi \rangle |.$$
First, in the case $\vec \xi=\vec I$, $\vec \xi$ is spacelike and
\begin{equation}
\label{eqn_psi}
e^{2 \psi} = \langle \vec \xi, \vec \xi\rangle = \frac{1}{H^2},
\end{equation}
so that $\Theta^L$ becomes:
\begin{align}
8\pi \Theta^L &=|\alpha_H|^2 + \frac{|\nabla^\Sigma H|^2}{H^2}, \label{eqn_Theta_L1}
\end{align}
having used (\ref{eqn_U1}) and 
$\langle \vec \xi, \vec H^{\perp} \rangle =0$.
In the case $\vec \xi=\beta \vec I^\perp$
$$
e^{2\psi}=  - \langle \vec \xi, \vec \xi\rangle = \frac{\beta^2}{H^2}$$
on the set where $\beta \neq 0$.  In particular, 
$$\psi = \log \beta - \log H.$$
The definition (\ref{defThetaL}) of  $\Theta^L$ becomes:
\begin{align}
8\pi \Theta^L &= -2\langle U, d\psi\rangle\langle \vec \xi, -\vec H^\perp\rangle 
\nonumber\\
&= 2 \beta \alpha_H \left(\frac{\nabla^\Sigma H}{H} - \frac{\nabla^\Sigma \beta}{\beta}\right) \nonumber\\
&= 2 \beta \alpha_H \left(\frac{\nabla^\Sigma H}{H}\right) - 2\alpha_H\left(\nabla^\Sigma \beta\right).\label{eqn_Theta_L2}
\end{align}
Note that (\ref{eqn_Theta_L2}) can be treated as the definition of $\Theta^L$, regardless of whether $\beta$ vanishes.

\begin{proof}[Alternate proof of Theorem \ref{thm:main}]
We combine the above computations with Theorem \ref{thm:BHMS}.  Let $Dm_H(\vec \xi)$ denote the derivative of the Hawking mass in the direction $\vec \xi$, evaluated on $\Sigma_0$. 

In the case $\vec \xi = \vec \xi_r = \vec I$, we use (\ref{eqn_Theta_T1}), (\ref{eqn_Theta_L1}), and $\langle \vec \xi,-\vec H^\perp\rangle=0$ to obtain:
\begin{equation*}
Dm_H(\vec \xi_r) = \sqrt{\frac{|\Sigma_0|}{(16\pi)^3}} \int_{\Sigma_0} \left[2G(-\vec H^\perp, \vec \xi_r^\perp)
+|\mathring\II_r|^2+|\mathring\II_t|^2 + 2|\alpha_H|^2 + \frac{2|\nabla^\Sigma H|^2}{H^2}\right].
\end{equation*}

In the case $\vec \xi = \vec \xi_t = \beta\vec I^\perp$, we use (\ref{eqn_U2}),(\ref{eqn_Theta_T2}), (\ref{eqn_Theta_L2}), and $\langle \vec \xi,-\vec H^\perp\rangle=-\beta$ to obtain:
\begin{align*}
Dm_H(\vec \xi_t) &= \sqrt{\frac{|\Sigma_0|}{(16\pi)^3}} \Bigg\{\int_{\Sigma_0} \left[2G(-\vec H^\perp, \vec \xi_t^\perp) + 2\beta \langle \mathring\II_{r}, \mathring\II_{t}\rangle
+4 \beta \alpha_H \left(\frac{\nabla^\Sigma H}{H}\right)\right] \\
&\qquad - \int_{\Sigma_0}
 \left[ 4\alpha_H\left(\nabla^\Sigma \beta\right) +2\beta \div_\Sigma(\alpha_H)\right] \Bigg\}\\
 &= \sqrt{\frac{|\Sigma_0|}{(16\pi)^3}} \Bigg\{\int_{\Sigma_0} \left[2G(-\vec H^\perp, \vec \xi_t^\perp) + 2\beta \langle \mathring\II_{r}, \mathring\II_{t}\rangle
+4 \beta \alpha_H \left(\frac{\nabla^\Sigma H}{H}\right)\right] \\
&\qquad + \int_{\Sigma_0}
 2\beta \div_\Sigma(\alpha_H) \Bigg\},
\end{align*}
having integrated by parts on the last line.
The formula now follows by adding $Dm_H(\vec \xi_r)$ and $Dm_H(\vec \xi_t)$ and using the linearity of $Dm_H(\cdot)$ and $G(-\vec H, \cdot)$.
\end{proof}

\section{Sufficient conditions for monotonicity}
\label{sec:monotonicity}

In this section we analyze conditions of the type $\div_{\Sigma} X =0$ that ensure
monotonicity of the Hawking mass for a uniformly area
expanding family of surfaces.  We start with the following lemma.

\begin{lemma}
\label{Sufficient}
Let $(\Sigma,h)$ be a closed Riemannian manifold. Let $\Psi,
\beta$ be scalar functions on $\Sigma$ with $\Psi >0$ and $|\beta| < 1$ 
and $X$  a vector field on $\Sigma$. If
\begin{eqnarray}
\divm_{\Sigma} ( \left ( G \circ \beta\right ) 
\left ( X - (V \circ \beta) \nabla^{\Sigma} \beta) \right ) =0 
\label{cond1}
\end{eqnarray}
where $G \in C^{\infty} ((-1,1), \mathbb{R}^+)$ 
and $V \in C^{\infty} ((-1,1),\mathbb{R})$
satisfies
\begin{eqnarray}
V(x) \left ( V(x) ( 1 - x^2 ) - 1 \right ) \geq 0 , \quad \quad x \in (-1,1),
\label{cond2}
\end{eqnarray}
then
\begin{eqnarray*}
\int_{\Sigma} 
 \left ( \frac{}{} |X|^2 + \left |\frac{\nablah \Psi}{\Psi} \right |^2 + 
2 \beta \left \langle X , \frac{\nablah \Psi}{\Psi} \right \rangle + \beta \divm_{\Sigma} (X)
\right )  \geq 0.
\end{eqnarray*}
\end{lemma}
After the proof we will apply this lemma to the last two lines of the main formula in Theorem \ref{thm:main}.

\begin{proof}
Define
\begin{eqnarray}
\F=  |X|^2 + \left |\frac{\nablah \Psi}{\Psi} \right |^2 + 
2 \beta \left \langle X , \frac{\nablah \Psi}{\Psi} 
\right \rangle  + \beta \div_{\Sigma} (X). 
\label{defA}
\end{eqnarray}
Let $x$ be the coordinate in $(-1,1)$ and prime
the derivative with respect to $x$. 
Define $B \in C^{\infty} ((-1,1), \mathbb{R}^{+})$ as any positive solution
of 
\begin{eqnarray}
\frac{B'}{B} + x V =0. \label{defB}
\end{eqnarray}
Introduce also a positive scalar function $\Phi$
and a vector field $Y$ by
\begin{align}
\Phi & = \frac{\Psi}{B \circ
\beta } \nonumber\\ 
Y & = X - \left ( V \circ \beta \right ) \nabla^{\Sigma} \beta.
\label{decomU} 
\end{align}
Inserting this decomposition 
into (\ref{defA}) yields, after using (\ref{defB}),
\begin{align*}
\F = & |Y|^2 + \left | \frac{\nablah \Phi}{\Phi} \right |^2
+ 2 \beta \left \langle Y , \frac{\nablah \Phi}{\Phi} \right \rangle
+ 2  \langle Y, \nablah \beta \rangle  \left ( ( 1 -x^2 ) V  \right )
\circ \beta +  \\
& 
+ \left | \nablah \beta  \right |^2
\left ( (1 - x^2 ) V^2 \right)
\circ \beta + \beta \div_{\Sigma} \left ( 
(V \circ \beta) \nablah \beta + Y \right ).
\end{align*}
Rewriting the last term gives:
\begin{align}
 \F = &  |Y|^2 + \left | \frac{\nablah \Phi}{\Phi} \right |^2
+ 2 \beta \left \langle Y , \frac{\nablah \Phi}{\Phi} \right \rangle
+  \langle Y , \nablah \beta \rangle \left ( - 1 + 2 \left (1 - x^2  \right )
V \right ) \circ \beta \nonumber \\
& + \left | \nablah \beta  \right |^2 \left  [ V \left (
V \left ( 1 - x^2 \right ) - 1 \right ) \right ] \circ \beta
+ \div_{\Sigma} \left ( \beta (V \circ \beta) \nablah \beta + \beta Y \right ).
\label{calA3}
\end{align}
Now, for any function $Q \in C^{\infty}( (-1,1), \mathbb{R})$ 
and vector field $Z$ on $\Sigma$, we have the immediate
identity
\begin{eqnarray}
(Q' \circ \beta) \langle Z,  \nablah \beta  \rangle =
\div_{\Sigma} \left ( (Q \circ \beta) Z   \right ) 
- (Q \circ \beta) \div_{\Sigma} (Z).
\label{identity1}
\end{eqnarray}
Let
$Z = \left ( G \circ \beta \right )  Y$, 
with $G \in C^{\infty} ( (-1,1), \mathbb{R}^+)$ so that
(\ref{identity1}) transforms into
\begin{eqnarray}
\label{identity2}
\left ( ( G Q' ) \circ \beta \right ) \langle Y , \nablah \beta \rangle
= \div_{\Sigma} \left ( ((G Q) \circ \beta) Y \right )
- (Q \circ \beta) \div_{\Sigma} \left ( (G \circ \beta) Y \right ) .
\end{eqnarray}
Choosing $Q$ to satisfy
\begin{eqnarray*}
Q' = \frac{1}{G} \left ( 1 - 2 (1 -x^2) V \right ),
\end{eqnarray*}
we can insert (\ref{identity2}) into (\ref{calA3}) to find
\begin{align*}
\F = & |Y|^2 + \left | \frac{\nablah \Phi}{\Phi} \right |^2
+ 2 \beta \left \langle  Y,\frac{\nablah \Phi}{\Phi} \right \rangle 
+ \div_{\Sigma} \left (  \beta (V \circ \beta) \nablah \beta 
- ( ( G Q - x ) \circ  \beta) Y \right )  \\
& + (Q \circ \beta) \div_{\Sigma} \left ( (G \circ \beta) Y
\right ) + \left | \nablah \beta  \right |^2 \left  ( V \left (
V ( 1 - x^2 ) - 1 \right )  \circ \beta \right ),
\end{align*}
from which it follows that conditions (\ref{cond1}) and (\ref{cond2}) imply
$\int_{\Sigma} \F  \geq 0$, by the divergence theorem.
\end{proof}

We can now apply this lemma to find sufficient conditions for
the monotonicity of the Hawking mass for a uniformly area expanding family of
surfaces. Recall first that for any
orthonormal basis $\{ \vec{\nu}, \vec{\nu}^{\perp} \}$ with $\vec \nu$
spacelike, the associated connection one-form is defined
as
\begin{eqnarray*}
\alpha_{\vec \nu} (X) = \langle \nabla_X \vec{\nu}, \vec{\nu}^{\perp} \rangle.
\end{eqnarray*}
Oriented orthonormal bases $\{ \vec{\nu}_{\theta}, \vec{\nu}_{\theta}^{\perp} \}
$ of the normal bundle, with outward-spacelike $\vec{\nu}_{\theta}$, are in 
one-to-one correspondence with smooth functions $\theta : \Sigma
\longrightarrow \mathbb{R}$ according to
\begin{equation*}
\vec{\nu}_{\theta} = \cosh \theta \, \vec{\nu}_H + \sinh \theta \, 
\vec{\nu}^{\bot}_H. 
\end{equation*}
Recall (equation (3.3) of \cite{TF}) that the connection one-form
$\alpha_{\vec \nu_\theta}$ relates to $\alpha_{H}$ as 
\begin{eqnarray}
\alpha_{\vec\nu_\theta} = \alpha_{H} - d \theta. \label{ConnectionTransform}
\end{eqnarray}

\begin{proposition}
\label{prop_suff}
Consider a  uniformly
area expanding family of surfaces with orthogonal flow vector
\begin{eqnarray}
\vec{\xi} = \vec{I} + \beta \vec{I}^{\perp}, \quad \quad |\beta| < 1
\label{vecxi} 
\end{eqnarray}
and let $\vec \nu_{\xi} = \frac{\vec{\xi}}{|\vec{\xi}|}$. The
Hawking mass is monotonic along this flow provided
$\divm_{\Sigma} (\alpha_{\vec \nu_{\Theta}} )=0$, 
where 
\begin{eqnarray} 
\vec{\nu}_{\Theta} = \cosh ( \Theta \circ \beta) \,  \vec{\nu}_H - 
\sinh (\Theta \circ \beta) \,
\vec{\nu}_H^{\bot} \label{case1}
\end{eqnarray}
or
\begin{eqnarray} 
\vec{\nu}_{\Theta} = \cosh (\Theta \circ \beta)) \, \vec{\nu_{\xi}}
+ \sinh (\Theta \circ \beta)
\, \vec{\nu}_{\xi}^{\bot} \label{case2}
\end{eqnarray}
and $\Theta \in C^{\infty} ((-1,1),\mathbb{R})$
is any nondecreasing function.
\end{proposition}

\begin{remark} The case $\Theta = 0$ in (\ref{case1}) corresponds to the 
time flat case of Theorem \ref{thm:main}. The case $\Theta =0$ in
(\ref{case2}), i.e. when $\div_{\Sigma} (\alpha_{\vec \nu_\xi}) =0$, was 
first discussed in \cite{MMS}. 
As follows from the proof below, the two sets 
(\ref{case1}) and (\ref{case2}) are disjoint. The only subcase that 
places no restrictions on $\beta$ is the time flat condition
$\div_\Sigma(\alpha_H)=0$.
\end{remark}

\begin{example}
Suppose in case (\ref{case1}) the function $\Theta(x)=x$ is chosen.  Then the hyperbolic angle from $\vec \nu_H$ to $\vec \nu_\Theta$ is $-\beta$.  By (\ref{ConnectionTransform}), $\alpha_{\vec \nu_\Theta} = \alpha_H + d\beta.$  The divergence-free condition that guarantees monotonicity reduces to the following Poisson
equation for $\beta$:
$$\Delta_\Sigma \beta = - \div_\Sigma(\alpha_H).$$
\end{example}

\begin{proof}[Proof of Proposition \ref{prop_suff}]
In view of the last two lines 
in the expression for the variation of the Hawking mass
in Theorem \ref{thm:main} we apply 
Lemma \ref{Sufficient} with  $X,\Psi$ defined by
$h(X,\cdot) = \alpha_H(\cdot)$ and $\Psi = H$.

For any $F \in C^{\infty} ((-1,1),\mathbb{R})$, we consider the
basis of normal vectors $\{ \vec{\nu}_{F}, \vec{\nu}_F^{\perp} \}$
defined by
\begin{equation}
\vec{\nu}_{F} = \cosh (F \circ \beta) \, \vec{\nu}_H + \sinh ( F \circ
\beta ) \vec{\nu}^{\bot}_H. \label{defnuF}
\end{equation}
The choice 
$V = F'$  implies $h(Y,\cdot) = \alpha_{\vec\nu_F}(\cdot)$ as a
consequence (\ref{ConnectionTransform}), where $Y$ is defined by (\ref{decomU}). Thus,
Lemma \ref{Sufficient} with  $G=1$ shows that the condition
\begin{eqnarray*}
\div_{\Sigma} \left ( \alpha_{\vec \nu_{F}} \right ) =0
\end{eqnarray*}
ensures monotonicity of the
Hawking mass mass
of a uniformly area expanding family of surfaces 
provided $F'$ satisfies
\begin{eqnarray*}
F'(x) \left ( F'(x) (1- x^2) - 1 \right )  \geq 0\quad \quad
x \in (-1,1).
\end{eqnarray*}
This is equivalent to  (i) $F'(x) \leq 0, \;\forall x \in (-1,1)$ or (ii) 
$F'(x) \geq (1- x^2)^{-1}, \;\forall x \in (-1,1)$. In case (i), let
$\Theta = - F$ and monotonicity for the class (\ref{case1}) is proved.
In case (ii), let $\Theta(x) = F(x) - \mbox{arctanh}(x) $, so that $\Theta' 
\geq 0$ and (\ref{defnuF}) becomes
\begin{equation*}
\vec{\nu}_{F} = \frac{\cosh (\Theta \circ \beta)}{\sqrt{1- \beta^2}} \, 
\left ( \vec \nu_H + \beta \vec \nu_H^{\perp} \right )
+ \frac{\sinh ( \Theta \circ \beta )}{\sqrt{1-\beta^2}} \, 
\left ( \vec \nu_H^{\perp} +  \beta  \vec \nu_H \right ).
\end{equation*}
From (\ref{vecxi}), $\vec \nu_\xi = \frac{1}{\sqrt{1-\beta^2}}
\left ( \vec \nu_H + \beta \vec \nu_H^{\perp} \right )$ 
which proves monotonicity for the class (\ref{case2}). 
\end{proof}

\end{document}